\documentclass[12pt]{article}

\usepackage{latexsym,amsfonts,amsmath,amssymb,epsfig,tabularx,amsthm,dsfont}

\theoremstyle{definition}

\newtheorem{thm}{Theorem} %[section]
\newtheorem{rem}[thm]{Remark}
\newtheorem{lem}[thm]{Lemma}

\newtheorem{cor}[thm]{Corollary}

\newcommand{\Z}{\mathbb{Z}}
\newcommand{\N}{\mathbb{N}}
 
\newcommand{\R}{\mathbb{R}}

\newcommand{\abs}[1]{\left\vert #1 \right\vert}	% absolute value / norm
\renewcommand{\l}{\langle}
\renewcommand{\r}{\rangle}
\renewcommand{\a}{\alpha}

\newcommand{\G}{\Gamma}
\newcommand{\e}{{\rm e}}

\renewcommand{\O}{\Omega}
\newcommand{\wt}{\widetilde}
\newcommand\dint{{\,\rm d}}
\newcommand{\mix}{{\rm mix}}
\newcommand{\vol}{{\rm vol}}
\newcommand{\supp}{{\rm supp}}

\renewcommand{\H}{H_d^{s,\mix}}
\newcommand{\Ho}{\mathring{H}_d^{s,\mix}}
\newcommand{\X}{\mathbb{X}}
\newcommand{\Qo}{\mathring{Q}}

\title{On ``Upper error bounds for quadrature formulas on function classes'' 
				by K.~K.~Frolov}   

\author{    
Mario Ullrich
}   

\begin{document}
\maketitle  
 
\begin{abstract} 
This is a tutorial paper that 
gives the complete proof of a result of Frolov~\cite{Fr76} that shows the optimal order 
of convergence for numerical integration of functions with bounded mixed 
derivatives. The presentation follows Temlyakov~\cite{Te03}, see also \cite{Te93}.
\end{abstract} 

\bigskip

\section{Introduction}

We study cubature formulas for the approximation of the 
$d$-dimensional integral
\[
I(f) \,=\, \int_{[0,1]^d} f(x) \dint x
\]
for functions $f$ with bounded mixed derivatives.
For this, let $D^\a f$, $\a\in\N_0^d$, be the usual partial derivative of a 
function $f$ and define the norm
\begin{equation} \label{eq:norm}
\|f\|_{s,\mix}^2 \,:=\, \sum_{\a\in\N_0^d:\, \|\a\|_\infty\le s} \|D^\a f\|_{L_2}^2,
\end{equation}
where $s\in\N$.
In the following we will study the class (or in fact the unit ball) 
\begin{equation} \label{eq:class}
H_d^{s,\mix} \,:=\, \overline{\bigl\{f\in C^{sd}([0,1]^d) \colon \|f\|_{s,\mix}\le1 \bigr\}}, 
\end{equation}
i.e. the closure in $C([0,1]^d)$ (with respect to $\|\cdot\|_{s,\mix}$) 
of the set of $sd$-times 
continuously differentiable functions $f$ with $\|f\|_{s,\mix}\le1$. 
Additionally, we will study the class
\begin{equation} \label{eq:class2}
\mathring{H}_d^{s,\mix}  \,:=\, \bigl\{f\in H_d^{s,\mix} \colon {\rm supp}(f)\subset (0,1)^d \bigr\}.
\end{equation}
%M Note that $\Ho$ contains only 1-periodic functions.

The algorithms under consideration are of the form
\begin{equation} \label{eq:algorithms}
Q_n(f) \,=\, \sum_{j=1}^n a_j f(x^j) 
\end{equation}
for a given set of nodes $\{x^j\}_{j=1}^{n}$, $x^j=(x^j_1,\dots,x^j_d)\in[0,1]^d$,
and weigths $(a_j)_{j=1}^{n}$, $a_j\in\R$, 
i.e.~the algorithm $A_n$ uses at most $n$ function evaluations of the input function.

The worst case error of $Q_n$ in the function class $H$ is defined as      
$$       
e(Q_n,H)=\sup_{f\in H} |I(f) - Q_n(f)|.      
$$       
%whereas the $n$th minimal worst case error is 
%\begin{equation}  \label{eq:error}      
%e(n, H) := \inf_{Q_n}\, e(Q_n, H).      
%\end{equation}     
%Here, the infimum is over all algorithms of the form \eqref{eq:algorithms}.
We will prove the following theorem.

\begin{thm} \label{thm1}
Let $s,d\in\N$. Then there exists a sequence of algorithms $(Q_n)_{n\in\N}$ 
such that
\[
e(Q_n, \Ho) \,\le\, C_{s,d}\; n^{-s}\, (\log n)^{\frac{d-1}{2}},
\]
where $C_{s,d}$ may depend on $s$ and $d$.
\end{thm}

Using standard techniques, see e.g.~\cite[Sec.~2.12]{SJ94} or \cite[Thm.~1.1]{Te03}, 
one can deduce (constructively) from the algorithm that is used to prove 
Theorem~\ref{thm1} a quadrature rule for $\H$ that has the same order of convergence. 
This results in the following corollary.

\begin{cor} \label{coro1}
Let $s,d\in\N$. Then there exists a sequence of algorithms $(Q_n)_{n\in\N}$ 
such that
\[
e(Q_n, \H) \,\le\, \wt C_{s,d}\; n^{-s}\, (\log n)^{\frac{d-1}{2}},
\]
where $\wt C_{s,d}$ may depend on $s$ and $d$.
\end{cor}

%Another corollary of Theorem~\ref{thm1} is 

The proof of Theorem~\ref{thm1}, and hence also of Corollary~\ref{coro1}, 
is constructive, i.e.~we will show how to construct the nodes and weights 
of the used algorithms. 
%The plan of the paper is as follows. In Section~\ref{sec:thm1} we present 
%the algorithm and the proof of the error bound for functions in $\Ho$ 
%from Theorem~\ref{thm1}. 
%In Section~\ref{sec:coro1} we construct an optimal algorithm for the class $\H$ 
%to prove Corollary~\ref{coro1}.
% and, finally, 
%in Section~\ref{sec:disc} we discuss the consequences of these 
%results for upper bounds on the $L_p$-discrepancy.

\begin{rem}\label{rem:lower_bounds}
The upper bounds of Theorem~\ref{thm1} and Corollary~\ref{coro1} that will 
be proven in the next section for a specific algorithm, see~\eqref{eq:algorithm1}, 
are best possible in the sense of the order of convergence. 
That is, there are matching lower bounds that hold for arbitrary 
quadrature rules that use only function values, see e.g.~\cite[Theorem~3.2]{Te03}.
\end{rem}

\begin{rem}
There is a natural generalization of the spaces $\Ho$, 
say $\mathring{H}^{s,{\rm mix}}_{d,p}$, 
where the $L_2$-norm in \eqref{eq:norm} is replaced by an $L_p$-norm, $1< p<\infty$. 
The same lower bounds as mentioned in Remark~\ref{rem:lower_bounds} are valid also in this case. 

Obviously, the upper bounds from Theorem~\ref{coro1} hold for these spaces 
if $p\ge2$, since the spaces get smaller for larger $p$. 
That the same algorithm satisfies the optimal order if $1<p<2$ was proven by 
Skriganov~\cite[Theorem~2.1]{Sk94}.
We refer to \cite{Te03} and references therein 
for more details on this, the more delicate case $p=1$, 
and the generalization to non-integer smoothness.
\end{rem}

\section{Proof of Theorem~\ref{thm1}} \label{sec:thm1}

\subsection{The algorithm for $\Ho$}

We start with the construction of the nodes of our quadrature rule. 
See Sloan and Joe~\cite{SJ94} for a more comprehensive introduction to this topic.
In the setting of Theorem~\ref{thm1} the set $X\subset[0,1)^d$ of nodes 
will be a subset of a \emph{lattice} $\X\subset\R^d$, 
i.e.~$x,y\in\X$ implies $x\pm y\in\X$. In fact, we take $X=\X\cap[0,1)^d$.

The lattice $\X$ will be ``$d$-dimensional''\footnote{It is well known that every 
lattice in $\R^d$ can be written as $T(\Z^m)$ for some $m\le d$ and some 
$d\times m$-matrix $T$ with linearly independent columns. 
The number $m$ is called the \emph{dimension} of the lattice.}, 
i.e.~there exists
a non-singular $d\times d$-matrix $T$ such that 
\begin{equation}\label{eq:lattice}
\X \,:=\, T(\Z^d) \,=\, \bigl\{Tx\colon x \in\Z^d\bigr\}. 
\end{equation}
The matrix $T$ is called the \emph{generator} of the lattice $\X$.
Obviously, every multiple of $\X$, i.e.~$c\X$ for some $c\in\R$, is again a lattice 
and note that while $\X$ is a \emph{lattice}, 
it is not necessarily an \emph{integration lattice}, 
i.e.~in general we do not have $\X\supset\Z^d$.

The nodes for our quadrature rule for functions from $\Ho$ will be all points 
inside the cube $[0,1)^d$ of the shrinked lattice $a^{-1}\X$, $a>1$. 
That is, we will use the set of points
\begin{equation} \label{eq:nodes}
X_a^d \,:=\, \bigl(a^{-1}\X\bigr) \cap [0,1)^d, \qquad a>1.
\end{equation}

For the construction of the nodes it remains to 
present a specific generator matrix $T$ that is suitable for our purposes. 
For this, define the polynomials 
\begin{equation} \label{eq:poly}
P_d(t) \,:=\, \prod_{j=1}^d \bigl(t-2j+1\bigr) -1, \qquad t\in\R. 
\end{equation}
Obviously, the polynomial $P_d$ has only integer coefficients, and it is 
easy to check that it is irreducible\footnote{A polynomial $P$ is called 
\emph{irreducible} over $\mathbb{Q}$ if $P=GH$ for two polynomials 
$G, H$ with rational coefficients implies that one of them has degree zero.
%if there are no two polynomials 
%$G, H$ with rational coefficients and degree greater one, such that $P=GH$.
In fact, every polynomial of the form $\prod_{j=1}^d(x-b_j)-1$ with 
different $b_j\in\Z$ is irreducible, but has not necessarily $d$ different real roots.}
(over $\mathbb{Q}$) and has $d$ different real roots.
Let $\xi_1,\dots,\xi_d\in\R$ be the roots of $P_d$.
Using these roots we define the $d\times d$-matrix $B$ by 
\begin{equation} \label{eq:generator_dual}
B \,=\, \bigl(B_{i,j}\bigr)_{i,j=1}^d \,:=\, \Bigl(\xi_i^{j-1}\Bigr)_{i,j=1}^d.
\end{equation}
This matrix is a Vandermonde matrix and hence invertible 
and we define the generator matrix of our lattice by
\begin{equation} \label{eq:generator}
T \,=\, (B^\top)^{-1},
\end{equation}
where $B^\top$ is the transpose of $B$. 
%Usually, $B(\Z^d)$ is called the \emph{dual lattice} associated with $\X=T(\Z^d)$.
It is well known that $\X^*:=B(\Z^d)$ is the \emph{dual lattice} associated with 
$\X=T(\Z^d)$, i.e.~$y\in\X^*$ if and only if $\l x,y\r\in\Z$ for all $x\in\X$.

We define the quadrature rule for functions $f$ from $\Ho$ by
\begin{equation}\label{eq:algorithm1}
\mathring{Q}_a(f) \,:=\, a^{-d} \det(T)\, 
%\left(a^d \det(T^{-1})\right)^{-1}\, 
	\sum_{x\in X_a^d} f(x), \qquad a>1.
\end{equation}
In the next subsection we will prove that $\mathring{Q}_a$ has the optimal order of 
convergence for $\Ho$. 

Note that $\Qo_a(f)$ uses $|X_a^d|$ function values of $f$ and 
that the weights of this algorithm are equal, but 
do not (in general) sum up to one. 
While the number $|X_a^d|$ of points can be estimated in terms of the 
determinant of the corresponding generator matrix, it is in general 
not equal.
In fact, if $a^{-1}\X$ would be an integration lattice, then it is well known 
that $|X_a^d|=a^d \det(T^{-1})$, see e.g.~\cite{SJ94}.
For the general lattices that we consider, we know, 
however, that these numbers are of the same order, see 
Skriganov~\cite[Theorem~1.1]{Sk94}\footnote{Skriganov proved this result 
for \emph{admissible} lattices. The required property will be proven in 
Lemma~\ref{lem:product}, see also \cite[Lemma~3.1(2)]{Sk94}.}. 

\begin{lem}\label{lem:number_of_nodes}
Let $\X=T(\Z^d)\subset\R^d$ be a lattice with generator $T$ of the 
form~\eqref{eq:generator}, 
and let $X_a^d$ be given by~\eqref{eq:nodes}. 
Then there exists a constant $C_T$ that is independent of $a$ such that
\[
\left| |X_a^d| - a^d \det(T^{-1}) \right| 
\;\le\; C_T\, \ln^{d-1}\bigl(1+a^d\bigr)
\]
for all $a>1$.
In particular, we have 
\[
\lim_{a\to\infty}\, \frac{|X_a^d|}{a^d \det(T^{-1})\,} = 1.
\]
\end{lem}

\begin{rem}
It is still not clear if the corresponding QMC algorithm, i.e.~the quadrature rule 
\eqref{eq:algorithm1} with $a^{-d} \det(T)$ replaced by $|X_a^d|^{-1}$, 
has the same order of convergence. 
In fact, if true, this would imply the optimal order of the $L_2$-discrepancy 
of a modification of the set $X_a^d$, see~\cite{Fr80}. 
We leave this as an open problem. 
%It is not clear (to me) if the algorithm with weights $|X_a^d|^{-1}$ has the same order 
%of convergence. 
%A proof of this could be done, e.g., with an improved version of 
%Lemma~\ref{lem:number_of_nodes} to bound $\abs{|X_a^d|^{-1}-a^{-d}\det(T)^{-1}}$. 
%I think (but I didn't check it carefully) it is easy to get the same order 
%in the case $s=1$ since $\abs{|X_a^d|^{-1}-a^{-d}\det(T)^{-1}}\ll |X_a^d|^{-1}$.
\end{rem}

In the remaining subsection we prove the crucial property of these nodes. 
For this we need the following corollary of the 
\emph{Fundamental Theorem of Symmetric Polynomials}, 
see,~\cite[Theorem~6.4.2]{FR97}.

\begin{lem} \label{lem:symmetric}
Let $P(x)=\prod_{j=1}^d(x-\xi_j)$ and $G(x_1,\dots,x_d)$ be polynomials 
with integer coefficients. 
Additionally, assume that $G(x_1,\dots,x_d)$ is symmetric in $x_1,\dots,x_d$, 
i.e.~invariant under permutations of $x_1,\dots,x_d$.
Then, $G(\xi_1,\dots,\xi_d)\in\Z$.
%In particular, $G(m)\in\Z$ for all $m\in\Z^k$.
\end{lem}

We obtain that the elements of the dual lattice $B(\Z^d)$ satisfy the following

\begin{lem} \label{lem:product}
Let $0\neq z=(z_1,\dots,z_d)\in B(\Z^d)$ with $B$ from \eqref{eq:generator_dual}. 
Then, $\prod_{j=1}^d z_i \in\Z\setminus0$.
\end{lem}

\begin{proof}
Fix $m=(m_1,\dots,m_d)\in\Z^d$ such that $Bm=z$.
Hence, 
\[
z_i \,=\, \sum_{j=1}^d m_j \xi_i^{j-1}
\]
depends only on $\xi_i$. This implies that $\prod_{j=1}^d z_i$ is a symmetric 
polynomial in $\xi_1,\dots,\xi_d$ with integer coefficients. 
By Lemma~\ref{lem:symmetric}, we have $\prod_{j=1}^d z_i\in\Z$. 

It remains to prove $z_i\neq0$ for $i=1,\dots,d$. 
Define the polynomial $R_1(x):=\sum_{j=1}^d m_j x^{j-1}$ and assume 
that $z_\ell=R_1(\xi_\ell)=0$ for some $\ell=1,\dots,d$. 
Then there exist unique polynomials $G$ and $R_2$ with rational coefficients 
such that
\[
P_d(x) \,=\, G(x) R_1(x) + R_2(x), 
\]
where ${\rm degree}(R_2)<{\rm degree}(R_1)$. By assumption, $R_2(\xi_\ell)=0$. 
If $R_2\equiv0$ this is a contradiction to the irreducibility of $P_d$. If not, 
divide $P_d$ by $R_2$ (instead of $R_1$). Iterating this procedure, we will 
eventually find a polynomial $R^*$ with ${\rm degree}(R^*)>0$ (since it has a root) 
and rational coefficients that divides $P_d$: a contradiction to the irreducibility.
This completes the proof of the lemma.
%note that the $\xi_i$'s are \emph{algebraic numbers} 
%(or even \emph{algebraic integers}), i.e.~the 
%roots of a polynomial with rational coefficients. Here, this polynomial is given 
%by $P_d$ from \eqref{eq:poly}. Since $P_d$ is irreducible (over $\mathbb{Q}$) 
%we know 
\end{proof}

We finish the subsection with a result on the maximal number of nodes in the dual 
lattice that lie in a axis-parallel box of fixed volume.

%\begin{cor} \label{coro:boxes}
%Let $B$ be the matrix from \eqref{eq:generator_dual}. 
%Then, for each axis-parallel box $\Omega\subset\R^d$ we have 
%\[
%\bigl\vert B(\Z^d)\cap \Omega \bigr\vert \,\le\, {\rm vol}_d(\Omega)+1.
%\]
%\end{cor}
%
%\begin{proof}
%Assume first that ${\rm vol}_d(\Omega)<1$. If $\Omega$ contains 2 different 
%points $z, z'\in B(\Z^d)$, then, using that this implies $z''=z-z'\in B(\Z^d)$, 
%we obtain 
%\[
%{\rm vol}_d(\Omega) \,\ge\, \prod_{i=1}^d |z_i-z'_i| 
%\,\ge\, \prod_{i=1}^d |z''_i| \,\ge\, 1
%\]
%from Lemma~\ref{lem:product}: a contradiction. 
%For ${\rm vol}_d(\Omega)\ge1$ we divide $\Omega$ along one coordinate in 
%$\lfloor {\rm vol}_d(\Omega)+1 \rfloor$ equal pieces and use the same 
%argument than above.
%\end{proof}

\begin{cor} \label{coro:boxes}
Let $B$ be the matrix from \eqref{eq:generator_dual} and $a>0$. 
Then, for each axis-parallel box $\Omega\subset\R^d$ we have 
\[
\bigl\vert aB(\Z^d)\cap \Omega \bigr\vert \,\le\, a^{-d}\,{\rm vol}_d(\Omega)+1.
\]
\end{cor}

\begin{proof}
Assume first that ${\rm vol}_d(\Omega)<a^d$. If $\Omega$ contains 2 different 
points $z, z'\in aB(\Z^d)$, then, using that this implies $z''=z-z'\in aB(\Z^d)$, 
we obtain 
\[
{\rm vol}_d(\Omega) \,\ge\, \prod_{i=1}^d |z_i-z'_i| 
\,\ge\, \prod_{i=1}^d |z''_i| \,\ge\, a^d
\]
from Lemma~\ref{lem:product}: a contradiction. 
For ${\rm vol}_d(\Omega)\ge a^d$ we divide $\Omega$ along one coordinate in 
$\lfloor a^{-d}\,{\rm vol}_d(\Omega)+1 \rfloor$ equal pieces, i.e.~pieces 
with volume less than $a^d$, and use the same argument as above.
\end{proof}

\begin{rem}
Although we focus in the following on the construction of nodes that is based 
on the polynomial $P_d$ from \eqref{eq:poly}, the same contruction works 
with any irreducible polynomial of degree $d$ with $d$ different real roots 
and leading coefficient 1. 
For example, if the dimension is a power of 2, i.e.~$d=2^k$ for 
some $k\in\N$, we can be even more specific. 
In this case we can choose the polynomial
\[
P^*_d(x) \,=\, 2 \cos\Bigl(d\cdot\arccos(x/2)\Bigr), 
\]
cf.~the \emph{Chebyshev polynomials}. The roots of this polynomial are given by 
\[
\xi_i \,=\, 2 \cos\left(\frac{\pi(2i-1)}{2d}\right), \qquad i=1,\dots,d.
\]
Hence, the construction of the lattice $\X$ that is based on this polynomial 
is completely explicit. 
For a suitable polynomial if $2d+1$ is prime, see \cite{LeeW11}.
We didn't try to find a completely explicit construction 
in the intermediate cases. 
\end{rem}

\subsection{The error bound} \label{subsec:error}

In this subsection we prove that the algorithm $\Qo_a$ from 
\eqref{eq:algorithm1} has the optimal order of convergence for functions 
from $\Ho$, i.e.~that 
\[
e(\Qo_a, \Ho) \,\le\, C_{s,d}\; n^{-s}\, (\log n)^{\frac{d-1}{2}}, 
\]
where $n=n(a,T):=|X_a^d|$ is the number of nodes used by $\Qo_a$ and 
$C_{s,d}$ is independent of $n$.

For this we need the following two lemmas. 
Recall that the \emph{Fourier transform} of an integrable function 
$f\in L_1(\R^d)$ is given by
\[
\hat f(y) \,:=\, \int_{\R^d} f(x)\, \e^{-2\pi\,i\,\l y,x\r} \dint x, \qquad y\in\R^d, 
\]
with $\l y,x\r:=\sum_{j=1}^d y_j x_j$. 
%Furthermore, let $\bar y=\prod_{j=1}^d y_j$.
% $\bar y=\prod_{j=1}^d \max\{1,|y_j|\}$
Furthermore, let
\begin{equation}\label{eq:multiplier}
\nu_s(y)=\prod_{j=1}^d \left(\sum_{\ell=0}^s |2\pi y_j|^{2\ell}\right)
= \sum_{\a\in\N_0^d:\, \|\a\|_\infty\le s}\, \prod_{j=1}^d |2\pi y_j|^{2\alpha_j}, 
\qquad y\in\R^d.
\end{equation}
Clearly, 
\[\begin{split}
\nu_s(y) |\hat{f}(y)|^2 
\,&=\, \sum_{\a\in\N_0^d:\, \|\a\|_\infty\le s}\,
			\left|\int_{\R^d} \prod_{j=1}^d (-2\pi\,i\, y_j)^{\alpha_j} f(x)\, 
				\e^{-2\pi\,i\,\l y, x\r} \dint x \right|^2\\
\,&=\, \sum_{\a\in\N_0^d:\, \|\a\|_\infty\le s}\, \abs{\widehat{D^\a f}(y)}^2
\end{split}\]
for all $f\in\H$ with compact support and $y\in\R^d$.

We begin with the following result on the sum of values of the Fourier transform.

\begin{lem}\label{lem:fourier_bound}
Let $f\in \H(\R^d)$ with $\supp(f)\subset A$ for some compact $A\subset\R^d$. 
Additionally, define 
$M_A:=\#\left\{m\in\Z^d\colon A\cap\bigl(m+(0,1)^d\bigr)\neq\varnothing\right\}$. 
%and let
%\begin{equation}\label{eq:multiplier}
%\nu_s(y)=\prod_{j=1}^d \left(\sum_{j=0}^s |2\pi y_j|^{2j}\right)
%= \sum_{\a\in\N_0^d:\, \|\a\|_\infty\le s}\, \prod_{j=1}^d |2\pi y_j|^{2\alpha_j}, 
%\qquad y\in\R^d.
%\end{equation}
Then, 
\[
\sum_{y\in\Z^d} \nu_s(y) |\hat{f}(y)|^2 
\,\le\, M_A\,\|f\|_{s,\mix}^2.
\]
%where $\nu(y)=\prod_{j=1}^d \left(\sum_{j=0}^s |2\pi y_j|^{2j}\right)$, $y\in\R^d$. 
\end{lem}

%\noindent{\bf M: I would like to replace $M_A$ by $\vol_d(A)$! Could this be true?}

\begin{proof}
Let $\G:=\{\a\in\N_0^d:\, \|\a\|_\infty\le s\}$. 
Define the function $g(x):=\sum_{m\in\Z^d} f(m+x)$, $x\in[0,1]^d$, 
and note that at most $M_A$ of the summands are not zero. 
Obviously, $g$ is 1-periodic.
Hence, we obtain by Parseval's identity and Jensen's inequality that
\[\begin{split}
\sum_{y\in\Z^d} \nu_s(y)|\hat{f}(y)|^2
%\,&=\, \sum_{\a\in\G}\sum_{y\in\Z^d} 
			%\left|\int_{\R^d} \prod_{j=1}^d (-2\pi\,i\, y_j)^{\alpha_j} f(x)\, 
				%\e^{-2\pi\,i\,\l y, x\r} \dint x \right|^2 \\
\,&=\, \sum_{\a\in\G}\sum_{y\in\Z^d} 
			\abs{\widehat{D^\a f}(y)}^2 
\,=\, \sum_{\a\in\G}\sum_{y\in\Z^d} \left|\int_{\R^d} D^\a f(x)\, 
				\e^{-2\pi\,i\,\l y, x\r} \dint x \right|^2 \\
\,&=\, \sum_{\a\in\G}\sum_{y\in\Z^d} \left|\sum_{m\in\Z^d} \int_{[0,1]^d} D^\a f(m+x)\, 
		\e^{-2\pi\,i\,\l y, x\r} \dint x  \right|^2 \\
\,&=\, \sum_{\a\in\G}\sum_{y\in\Z^d} \left|\int_{[0,1]^d} D^\a g(x)\, 
		\e^{-2\pi\,i\,\l y, x\r} \dint x  \right|^2 
=\, \sum_{\a\in\G} \int_{[0,1]^d} \left|D^\a g(x)\right|^2 \dint x \\
\,&=\, M_A^2\,\sum_{\a\in\G}\int_{[0,1]^d} 
		\left|\frac{1}{M_A}\sum_{m\in\Z^d} D^\a f(m+x)\right|^2 \dint x \\
\,&\le\, M_A^2\,\sum_{\a\in\G}\int_{[0,1]^d} \frac{1}{M_A}\sum_{m\in\Z^d} \left|D^\a f(m+x)\right|^2 \dint x \\
\,&=\, M_A\,\sum_{\a\in\G} \int_{\R^d} \left|D^\a f(x)\right|^2 \dint x 
\,=\, M_A\,\|f\|_{s,\mix}^2
\end{split}\] 
as claimed.\\
\end{proof}

Additionally, we need the following 
version of the \emph{Poisson summation formula} for lattices.

\begin{lem}\label{lem:poisson}
Let $\X=T(\Z^d)\subset\R^d$ be a full-dimensional lattice 
and $\X^*\subset\R^d$ be the associated dual lattice. 
Additionally, let $f\in\Ho$, $s\ge1$. Then,
\[
\det(T)\sum_{x\in\X\cap[0,1)^d} f(x) \,=\, \sum_{y\in\X^*} \hat{f}(y).
\]
In particular, the right-hand-side is convergent.
\end{lem}

\begin{proof}
%Let $\bar{f}:\R^d\to\R$ be the function with $\bar{f}(x)=f(x)$ for $x\in[0,1]^d$ 
%and $\bar{f}(x)=0$ for $x\notin[0,1]^d$. 
To ease the notation we identify each $f\in\Ho$ with the continuation 
to the whole space by zero, i.e.~$f(x)=0$ for $x\notin[0,1]^d$. 
Let $g(x)=f(Tx)$, $x\in\R^d$. 
Then, by the definition of the lattice, we have
\[
\sum_{x\in\X\cap[0,1)^d} f(x) \,=\, \sum_{x\in\X} f(x) 
\,=\, \sum_{x\in\Z^d} f(Tx)
\,=\, \sum_{x\in\Z^d} g(x).
\]
Additionally, note that $B=(T^\top)^{-1}$ is the generator of $\X^*$ and hence  
\[\begin{split}
\sum_{y\in\X^*} \hat{f}(y) \,&=\, \sum_{y\in\Z^d} \hat{f}(By)
\,=\, \sum_{y\in\Z^d} \int_{\R^d} f(x)\, \e^{-2\pi\,i\,\l By,x\r} \dint x
\,=\, \sum_{y\in\Z^d} \int_{\R^d} f(x)\, \e^{-2\pi\,i\,\l y, B^\top x\r} \dint x \\
\,&=\, \det(T)\,\sum_{y\in\Z^d} \int_{\R^d} f(Tz)\, \e^{-2\pi\,i\,\l y, z\r} \dint z
\,=\, \det(T)\,\sum_{y\in\Z^d} \int_{\R^d} g(z)\, \e^{-2\pi\,i\,\l y, z\r} \dint z\\
\,&=\, \det(T)\,\sum_{y\in\Z^d} \hat{g}(y), 
\end{split}\]
where we performed the substitution $x=Tz$. 
(Here, we need that the lattice is full-dimensional.)
In particular, the series on the left hand side converges if and only if 
the right hand side does. 
%For the proof of this convergence let 
%$h:=g^{(1,\dots,1)}=\frac{\partial^{d}g}{\partial x_1\dots\partial x_d}$ and 
%note that $f\in\Ho$, $s\ge1$, implies $\|h\|_{L_2(\R^d)}\le\|g\|_{s,\mix}<\infty$.
%%Moreover, the function $h^*(x):=\sum_{m\in\Z^d} h(m+x)$, $x\in[0,1]^d$, 
%%is periodic in $[0,1]^d$
%We obtain by Lemma~\ref{lem:fourier_bound} that 
%\[
%\sum_{y\in\Z^d} |\hat{h}(y)|^2 \,\le\, M_{T^{-1}([0,1]^d)}\,\|h\|_{L_2(\R^d)}^2 < \infty
%\]
%with $M_{T^{-1}([0,1]^d)}$ from Lemma~\ref{lem:fourier_bound}, since 
%$\supp(h)\subset T^{-1}([0,1]^d)$.
%%by Parseval's identity that
%%\[\begin{split}
%%\sum_{y\in\Z^d} |\hat{h}(y)|^2
%%\,&=\, \sum_{y\in\Z^d} \left|\int_{\R^d} h(x)\, \e^{-2\pi\,i\,\l y, x\r} \dint x \right|^2
%%\,=\, \sum_{y\in\Z^d} \left|\sum_{m\in\Z^d} \int_{[0,1]^d} h(m+x)\, 
		%%\e^{-2\pi\,i\,\l y, x\r} \dint x  \right|^2 \\
%%&=\, \int_{[0,1]^d} \left|\sum_{m\in\Z^d} h(m+x)\right|^2 \dint x 
%%\,=\,  \int_{\R^d} \left|h(x)\right|^2 \dint x 
%%\,=\, \|h\|_{L_2(\R^d)}^2 < \infty
%%\end{split}\] 
%Hence, %with $\bar y=\prod_{j=1}^d \max\{1,|y_j|\}$, 
%by the differentiation property of the Fourier transform,
%\[
%\sum_{y\in\Z^d} |\hat{g}(y)| 
%\,=\, \hat{g}(0)+\sum_{y\neq0} |\bar{y}|^{-1}\, |\hat{h}(y)| 
%\,\le\, \hat{g}(0)+\left(\sum_{y\neq0} |\bar{y}|^{-2}\right)^{1/2}\, 
			%\left(\sum_{y\neq0} |\hat{h}(y)|^2 \right)^{1/2} < \infty,
%\]
%which proves the convergence.
For the proof of this convergence 
%let $h:=g^{(1,\dots,1)}=\frac{\partial^{d}g}{\partial x_1\dots\partial x_d}$ and 
note that $f\in\Ho$, $s\ge1$, implies $\|g\|_{1,\mix}\le\|g\|_{s,\mix}<\infty$.
We obtain by Lemma~\ref{lem:fourier_bound} that 
\[
\sum_{y\in\Z^d}\nu_1(y) |\hat{g}(y)|^2 \,\le\, M_{T^{-1}([0,1]^d)}\,\|g\|_{1,\mix}^2 < \infty
\]
with $M_{T^{-1}([0,1]^d)}$ from Lemma~\ref{lem:fourier_bound}, since 
$\supp(g)\subset T^{-1}([0,1]^d)$.
Hence, 
\[
\sum_{y\in\Z^d} |\hat{g}(y)| 
\,\le\, \left(\sum_{y\neq0} |\nu_1(y)|^{-1}\right)^{1/2}\, 
			\left(\sum_{y\neq0} \nu_1(y)\,|\hat{g}(y)|^2 \right)^{1/2} < \infty,
\]
which proves the convergence.
We finish the proof of Lemma~\ref{lem:poisson} by
\[\begin{split}
\sum_{y\in\Z^d} \hat{g}(y) 
\,=\, \sum_{y\in\Z^d} \int_{\R^d} g(z)\, \e^{-2\pi\,i\,\l y, z\r} \dint z
\,=\, \sum_{y\in\Z^d} \int_{[0,1]^d} 
		\sum_{m\in\Z^d} g(m+z)\, \e^{-2\pi\,i\,\l y, z\r} \dint z 
\,=\, \sum_{m\in\Z^d} g(m). 
\end{split}\] 
The last equality is simply the evaluation of the Fourier series of the 
function $\sum_{m\in\Z^d} g(m+x)$, $x\in[0,1]^d$, at the point $x=0$. 
It follows from the absolute convergence of the left hand side 
that this Fourier series is pointwise convergent.
\end{proof}

%Now we are able to prove Theorem~\ref{thm1}.

%\begin{proof}[Proof of Theorem~\ref{thm1}]

By Lemma~\ref{lem:poisson} we can write the algorithm $\Qo_a$, $a>1$, as
\[
\Qo_a(f) \,=\,a^{-d} \det(T)\, \sum_{x\in X_a^d} f(x) 
%\,=\, \sum_{m\in\Z^d} \hat{f}(aBm), \qquad f\in\Ho,
\,=\, \sum_{z\in aB(\Z^d)} \hat{f}(z), \qquad f\in\Ho,
\]
where $aB$ (see \eqref{eq:generator_dual}) is the 
generator of the dual lattice of $a^{-1}T(\Z^d)$ (see \eqref{eq:generator}) 
and $X_a^d=(a^{-1}\X)\cap[0,1)^d$.
%(Recall that the generator of $X_a^d$ is given by $T/a$.)
Since $I(f)=\hat{f}(0)$ we obtain
\[\begin{split}
|I(f)-\Qo_a(f)| \,&=\, \abs{\sum_{z\in aB(\Z^d)\setminus0}\hat{f}(z)}
\,\le\, \sum_{z\in aB(\Z^d)\setminus0}\abs{\nu_s(z)}^{-1/2}\abs{\nu_s(z)^{1/2}\,\hat{f}(z)}\\
\,&\le\, \left(\sum_{z\in aB(\Z^d)\setminus0} |\nu_s(z)|^{-1}\right)^{1/2}\, 
			\left(\sum_{z\in aB(\Z^d)\setminus0} \nu_s(z)\, |\hat{f}(z)|^2 \right)^{1/2}.
\end{split}\]
with $\nu_s$ from \eqref{eq:multiplier}.
We bound both sums separately. 
First, note that Lemma~\ref{lem:fourier_bound} implies that
\[\begin{split}
\sum_{z\in aB(\Z^d)\setminus0} \nu_s(z)\, |\hat{f}(z)|^2
\,&=\, \sum_{m\in\Z^d\setminus0} \nu_s(aBm)\, |\hat{f}(aBm)|^2 
\,=\, \sum_{\a:\, \|\a\|_\infty\le s}\,\sum_{m\in\Z^d\setminus0}\, 
			\abs{\widehat{D^\a f}(aBm)}^2 \\
\,&=\, \sum_{\a:\, \|\a\|_\infty\le s}\,\sum_{m\in\Z^d\setminus0} \abs{\int_{\R^d} D^\a f(x)\, \e^{-2\pi\,i\,\l aBm, x\r} \dint x}^2 \\
\,&=\, (a^{-d}\det(T))^2\sum_{\a:\, \|\a\|_\infty\le s}\,\sum_{m\in\Z^d\setminus0} 
			\abs{\int_{\R^d} D^\a f(a^{-1}Ty)\, \e^{-2\pi\,i\,\l m, y\r} \dint y}^2 \\
\,&\le\, (a^{-d}\det(T))^2\, M_{a T^{-1}([0,1]^d)}\sum_{\a:\, \|\a\|_\infty\le s}\,\int_{\R^d} \abs{D^\a f(a^{-1}Ty)}^2 \dint y \\
\,&=\, C(a,T)\, \|f\|_{s,\mix}^2 
\end{split}\]
with $C(a,T):=a^{-d}\det(T)\, M_{a T^{-1}([0,1]^d)}$.
%Additionally, note that $M_{a T^{-1}([0,1]^d)}=|X_a^d|$ and hence, 
%by Lemma~\ref{lem:number_of_nodes}
Using that $T^{-1}([0,1]^d)$ is Jordan measurable, we obtain 
$\lim_{a\to\infty}C(a,T)=1$ and, hence, for $a>1$ large enough, 
\begin{equation} \label{eq:term2}
\left(\sum_{z\in aB(\Z^d)\setminus0} \nu_s(z)\, |\hat{f}(z)|^2 \right)^{1/2}
\,\le\, 2 \|f\|_{s,\mix}.
\end{equation}

Now we treat the first sum. 
Recall from Lemma~\ref{lem:product} that $\prod_{j=1}^d{z_j}\in\Z\setminus0$ for all 
$z\in B(\Z^d)\setminus0$ and define, for $m=(m_1,\dots,m_d)\in\N_0^d$, the sets 
\[
\rho(m) \,:=\, \{x\in\R^d\colon \lfloor2^{m_j-1}\rfloor\le |x_j|<2^{m_j} 
\text{ for } j=1,\dots,d\}.
\]
Note that $\prod_{j=1}^d|x_j|<2^{\|m\|_1}$ for all $x\in\rho(m)$. 
This shows that $|(aB(\Z^d)\setminus0)\cap\rho(m)|=0$ for all 
$m\in\N_0^d$ with $\|m\|_1\le \lfloor d\log_2(a)\rfloor=:r$.
Hence, with $|\bar{z}|:=\prod_{j=1}^d\max\{1,2\pi|z_j|\}$, we obtain 
\[\begin{split} 
\sum_{z\in aB(\Z^d)\setminus0} |\nu_s(z)|^{-1}
\,&\le\, \sum_{z\in aB(\Z^d)\setminus0} |\bar{z}|^{-2s}
\,=\, \sum_{\ell=r+1}^\infty\,\sum_{m: \|m\|_1=\ell}\,
			\sum_{z\in (aB(\Z^d)\setminus0)\cap\rho(m)} |\bar{z}|^{-2s}.
\end{split}\]
Note that for $z\in\rho(m)$ we have 
$|\bar{z}|\ge \prod_{j=1}^d\max\{1,2\pi\lfloor2^{m_j-1}\rfloor\}\ge2^{\|m\|_1}$.
Since $\rho(m)$ is a union of $2^d$ axis-parallel boxes 
each with volume less than $2^{\|m\|_1}$, 
Corollary~\ref{coro:boxes} implies that 
$\abs{(aB(\Z^d)\setminus0)\cap\rho(m)}\le 2^d(a^{-d} 2^{\|m\|_1}+1)
\le 2^{d+1} a^{-d} 2^{\|m\|_1}\le 2^{d+2} 2^{\|m\|_1-r}$ 
for $m$ with $\|m\|_1\ge r$.
Additionally, note that 
$\abs{\{m\in\N_0^d\colon \|m\|_1=\ell\}}=\binom{d+\ell-1}{\ell}<(\ell+1)^{d-1}$.
We obtain 
\[\begin{split}
\sum_{z\in aB(\Z^d)\setminus0} |\nu_s(z)|^{-1}
\,&\le\, \sum_{\ell=r+1}^\infty\sum_{m: \|m\|_1=\ell}\,
			\abs{(B(\Z^d)\setminus0)\cap\rho(m)}\, 2^{-2s \|m\|_1}\\
\,&\le\, 2^{d+2}\, \sum_{\ell=r+1}^\infty\sum_{m: \|m\|_1=\ell}\, 
				2^{\|m\|_1-r}\, 2^{-2s\|m\|_1} \\
\,&\le\, 2^{d+2}\,\sum_{\ell=r+1}^\infty (\ell+1)^{d-1}\,
				2^{\ell-r}\, 2^{-2s\ell} 
\,=\, 2^{d+2}\, \sum_{t=1}^\infty (t+r+1)^{d-1}\,
				2^{t}\, 2^{-2s(t+r)} \\
\,&\le\, 2^{2d+2}\, 2^{-2sr}\,r^{d-1}\,\sum_{t=1}^\infty (t+1)^{d-1}\,2^{(1-2s)t}\\
\,&\le\, 2^{2d+2s+2}\, a^{-2sd}\,\log_2\bigl(a^d\bigr)^{d-1}\;
				\sum_{t=1}^\infty (t+1)^{d-1}\,2^{(1-2s)t},\\
\end{split}\]
where we've used that $d\log_2(a)-1\le r\le d\log_2(a)$. 
Note that the last sum equals $2^{2s-1} {\rm Li}_{1-d}(2^{1-2s})-1$, where {\rm Li} 
is the \emph{polylogarithm} (also known as Jonqui\`ere's function), i.e. 
 \[
{\rm Li}_s(z) \,:=\, \sum_{\ell=1}^\infty \frac{z^\ell}{\ell^s}.
\]

So, all together 
\begin{equation} \label{eq:explicit_bound}
e(\Qo_a, \Ho) \,\le\, c_{s,d}\; a^{-sd}\, \left(\log_2\bigl(a^d\bigr)\right)^{\frac{d-1}{2}}
\end{equation}
for $a>1$ large enough, where 
\[
c_{s,d} \,=\, 2^{d+2s+1}\, {\rm Li}_{1-d}(2^{1-2s})^{1/2}.
\]
From Lemma~\ref{lem:number_of_nodes} we know that the number of nodes 
used by $\Qo_a$ is proportional to $a^d$. This proves Theorem~\ref{thm1}.

\begin{rem} \label{rem:domain}
It is interesting to note that the proof of Theorem~\ref{thm1} is to 
a large extent independent of the domain of integration. 
For an arbitrary Jordan measurable set $\O\subset\R^d$ we can consider the 
algorithm $\Qo_a$ from \eqref{eq:algorithm1} with 
the set of nodes $X_a^d$ replaced by 
$X_a^d(\O)=\bigl(a^{-1}T(\Z^d)\bigr) \cap\O$. 
The only difference in the estimates would be that $C(a,T)$, cf.~\eqref{eq:term2}, 
converges to $\vol_d(\O)$ instead of 1.
Thus, we have 
\[
e(\Qo_a, \Ho(\O)) 
\,\le\, c_{s,d,\O}\; a^{-sd}\, \left(\log_2\bigl(a^d\bigr)\right)^{\frac{d-1}{2}}
\]
for large enough $a>1$ with $c_{s,d,\O}=c_{s,d}\, \vol_d(\O)^{1/2}$ 
as in \eqref{eq:explicit_bound}.
Since $a^d$ is in this case proportional to 
$\abs{X_a^d(\O)}/(\det(T^{-1}) \vol_d(\O))$,  
there would be some additionaly volume dependent terms in $C_{s,d}$ from 
Theorem~\ref{thm1}. 
\end{rem}

\end{document}